\documentclass[12pt]{amsart}
\usepackage{amssymb}
\usepackage{epsfig}
\usepackage{mathdots}
\usepackage{epic,eepic,amsmath,amsthm,amssymb}
\usepackage{epsfig,cite,indentfirst}
\usepackage{multicol,boxedminipage}
\usepackage[textheight=22.0cm]{geometry}
\usepackage{pgf}
 \theoremstyle{plain}
\newtheorem{theorem}{Theorem}
\newtheorem{corollary}{Corollary}

\newtheorem{lemma}{Lemma}
\newtheorem{proposition}{Proposition}
\newtheorem{example}{Example}
\theoremstyle{definition}
\newtheorem{definition}{Definition}
\theoremstyle{remark}

\numberwithin{equation}{section}

\newcommand{\bT}{\begin{theorem}}
\newcommand{\eT}{\end{theorem}}
\newcommand{\bProp}{\begin{proposition}}
\newcommand{\eProp}{\end{proposition}}
\newcommand{\bE}{\begin{example}}
\newcommand{\eE}{\end{example}}
\newcommand{\bL}{\begin{lemma}}
\newcommand{\eL}{\end{lemma}}
\newcommand{\bP}{\begin{proof}}
\newcommand{\eP}{\end{proof}}
\newcommand{\bC}{\begin{corollary}}
\newcommand{\eC}{\end{corollary}}
\newcommand{\bD}{\begin{definition}}
\newcommand{\eD}{\end{definition}}
\newcommand{\be}{\begin{enumerate}}
\newcommand{\ee}{\end{enumerate}}
\newcommand{\beqa}{\begin{eqnarray*}}
\newcommand{\eeqa}{\end{eqnarray*}}
\newcommand{\beqaa}{\begin{eqnarray}}
\newcommand{\eeqaa}{\end{eqnarray}}
\newcommand{\ba}{\begin{array}}
\newcommand{\ea}{\end{array}}

\setbox0=\hbox{$+$}
\newdimen\plusheight
\plusheight=\ht0
\def\+{\;\lower\plusheight\hbox{$+$}\;}

\setbox0=\hbox{$-$}
\newdimen\minusheight
\minusheight=\ht0
\def\-{\;\lower\minusheight\hbox{$-$}\;}

\setbox0=\hbox{$\cdots$}
\newdimen\cdotsheight
\cdotsheight=\plusheight
\def\cds{\lower\cdotsheight\hbox{$\cdots$}}

\newtheorem{ca}{Figure}

\renewcommand{\mod}{\textup{mod}\,}

\begin{document}

\title[Hybrid Proofs of the $q$-Binomial Theorem]{Hybrid Proofs of the $q$-Binomial Theorem and other identities}

\author{Dennis Eichhorn}
 \address{Department of Mathematics
 University of California, Irvine, Irvine, CA 92697-3875 }
 \email{deichhor@math.uci.edu}
 
\author{James Mc Laughlin}
\address{Mathematics Department\\
West Chester University, West Chester, PA 19383}
\email{jmclaughlin2@wcupa.edu}

\author{Andrew V. Sills}
\address{Department of Mathematical Sciences\\
 Georgia Southern University\\
Statesboro, GA; telephone 912-681-5892; fax 912-681-0654}
\email{asills@georgiasouthern.edu}


\date{Sept 10, 2010}

\maketitle

\begin{abstract}
We give  ``hybrid" proofs of the $q$-binomial theorem and other identities. The proofs are ``hybrid"  in the sense that we use partition arguments to prove a restricted version of the theorem, and then use analytic methods (in the form of the Identity Theorem) to prove the full version.

We prove three somewhat unusual summation formulae, and use these to give hybrid proofs of a number of identities due to Ramanujan.

Finally, we use these new summation formulae to give new partition interpretations of the Rogers-Ramanujan identities and the Rogers-Selberg identities.
\end{abstract}

\maketitle

\setcounter{section}{0}

\section{Introduction}\label{introduction}

The proof of a $q$-series identity, whether a series-to-series
identity such as the second iterate of Heine's transformation (see
\eqref{heine2} below), a basic hypergeometric summation formula such
as the $q$-Binomial Theorem
  (see \eqref{bintheoeq0}) or one of the Rogers-Ramanujan identities
  (see \eqref{rr1} below), generally falls into one of two broad camps.

In the one camp, there are a variety of analytic methods. These
include (but are certainly not limited to) elementary $q$-series
manipulations (as in the proof of the Bailey-Daum summation formula
on page 18 of \cite{GR04}), the use of difference operators (as in
Gasper and Rahman's derivation of a bibasic summation formula
\cite{GR90}), the use of Bailey pairs and WP-Bailey pairs (see, for
example, \cite{AB02, S52, W03}), determinant methods (for example,
\cite{GP09, S17}), constant term methods (such as in \cite[Chap.
4]{A86}), polynomial finitization/generalization of infinite
identities (as in \cite{S03}), an extension of Abel's Lemma (see
\cite[Chap. 7]{AB09}), algorithmic methods such as the
$q$-Zeilberger algorithm (as in \cite{CHM08, K93}), matrix
inversions (including those of Carlitz \cite{C73} and Krattenthaler
\cite{K96}), $q$-Lagrange inversion (see \cite{A75, GS83}), Engel
expansions (see \cite{AKK00,AKP00})
 and several  other classical methods, including
 ``Cauchy's Method" \cite{JS05} and Abel's lemma on
 summation by parts \cite{C07}.

 In the other camp there are a variety of combinatorial or bijective proofs.
 Rather than attempt any classification of the various bijective proofs, we refer the reader to Pak's excellent survey \cite{P06} of bijective methods, with its extensive bibliography.

 In the present paper we use a ``hybrid" method to prove a number of basic hypergeometric identities. The proofs are ``hybrid"  in the sense that we use partition arguments to prove a restricted version of the theorem, and then use analytic methods (in the form of the Identity Theorem) to prove the full version.

We also prove three somewhat unusual summation formulae, and use these to give hybrid proofs of a number of identities due to Ramanujan.
Finally, we use these new summation formulae to give new partition interpretations of the Rogers-Ramanujan identities and the Rogers-Selberg identities.

\section{A Hybrid Proof of the $q$-Binomial Theorem}

In this section we give a hybrid proof of the $q$-Binomial Theorem,
\begin{equation}\label{bintheoeq0}
\sum_{n=0}^{\infty}\frac{(a;q)_nz^n}{(q;q)_n}
=\frac{(az;q)_{\infty}}{(z;q)_{\infty}}.
\end{equation}

\begin{lemma}
Let $k \geq 4$ and $r, s$ be fixed positive integers with $0<r<s<r+s<k$.
For each positive integer $n$ and each integer $m \geq (r+k)n$, let
$A_n(m)$ denote the number of partitions of $m$ with
\begin{itemize}
 \item the part $r$ occurring exactly $n$ times,
 \item distinct parts from  $\{s,s+k,s+2k,\dots, s+(n-1)k\}$,
 \item possibly repeating parts from $\{k,2k,3k,\dots , nk\}$, with the part $nk$ occurring at least once.
  \end{itemize}
Likewise, let $B_n(m)$ denote the number of partitions of $m$ into exactly $n$ parts, with
\begin{itemize}
  \item distinct parts $\equiv r+s (\mod k)$, with the part $r+s$ not appearing,
 \item possibly repeating parts $\equiv r (\mod k)$, with the part $r$ not appearing.
  \end{itemize}
  Then
  \[
  A_n(m)=B_n(m).
  \]
\end{lemma}

\begin{proof}
We will exhibit injections between the two sets of partitions. We may represent a partition of $m$ of the type counted by $A_n(m)$ as
\[
m = \sum_{j=1}^n m_j (jk)+\sum_{j=0}^{n-1} \delta_j(jk+s)+n(r),
\]
where the parts are displayed in parentheses, and the multiplicities satisfy $m_n\geq 1$, $m_j\geq 0$ for $1\leq j \leq n-1$, and $\delta_j \in \{0,1\}$. Upon applying the identity $\sum_{j=1}^t jy_j=\sum_{j=1}^t\sum_{i=j}^ty_i$ to the sums containing $j$, we get
\begin{equation*}
m =\left(m_n k+\delta_{n-1}s+r \right)
+
\sum_{j=1}^{n-1}\left( k\sum_{i=j}^n m_i + k\sum_{i=j}^{n-1}\delta_i +\delta_{j-1}s +r\right).
\end{equation*}
Here the parts of the new partition are displayed inside parentheses, and it is not difficult to recognize this partition as one of the type counted by $B_n(m)$.

On the other hand, we may represent a partition of $m$ of the type counted by $B_n(m)$ as
\[
m=\sum_{j=1}^n\left(p_j k +\delta_j s +r \right)
\]
with $1\leq p_1\leq p_2 \leq \dots \leq p_n$, $\delta_j\in\{0,1\}$, and if $\delta_i=\delta_{i+1}=1$, then $p_i<p_{i+1}$. We also label the $p_j$ so that if $p_i k+\delta_i s +r>p_j k+\delta_j s+r$, then $i>j$ (in particular, this labeling means $p_{j+1}-p_j-\delta_j\geq 0$ for $1 \leq j \leq n-1$). We rewrite the above sum for $m$ as
{\allowdisplaybreaks\begin{align*}
m&=n[r]+\delta_n[s]\\
&+(p_n-p_{n-1}-\delta_{n-1})[k]+\delta_{n-1}[k+s]\\
&+(p_{n-1}-p_{n-2}-\delta_{n-2})[2k]+\delta_{n-2}[2k+s]\\
&+(p_{n-2}-p_{n-3}-\delta_{n-3})[3k]+\delta_{n-3}[3k+s]\\
&\,\,\vdots\\
&+(p_3-p_2-\delta_2)[(n-2)k]+\delta_2[(n-2)k+s]\\
&+(p_2-p_1-\delta_1)[(n-1)k]+\delta_1[(n-1)k+s]\\
&+p_1[nk].
\end{align*}}
This is a partition of the type counted by $A_n(m)$, where this time the parts are displayed inside $[\,]$'s.

It is not difficult to see that these transformations give injections between the two sets of partitions and the result is proved.
\end{proof}

Graphically, we may describe these transformations as follows. In
each case, we start with the usual Ferrers diagram of the partition.

It can be seen that the largest part in a partition counted by $A_n(m)$ has size $nk$, so such a partition can be regarded as consisting of $n$ columns, each of width $k$. The first step is to distribute the $n$ parts of size $r$ so that one $r$ is at the bottom of each of these $n$ columns. We then form a new partition whose parts are the columns of this intermediate partition (we might call it \emph{the $k$-block conjugate} of this partition). This new partition is easily seen to be a partition of the type counted by $B_n(m)$.

If we start with a partition of the type counted by $B_n(m)$, the first step is to strip away a part of size $r$ from each of the $n$ parts. We then form the $k$-block conjugate of the remaining partition, add in the $n$ parts of size $r$, and what results is a partition of the type counted by $A_n(m)$.

We illustrate these transformations with two partitions of $26k+4s+5r$ (with $n=5$). The partition with parts
$5k, 4k+s, 4k, 4k, 3k+s, 2k, 2k, k+s, k, s, r,r,r,r,r$ is one of those counted by $A_5(26k+4s+5r)$. Its Ferrers diagram follows, and we show how it is transformed into the partition with parts $9k+s+r, 7k+s+r, 5k+r, 4k+r+s$ and $k+r+s$, which is a partition of the type counted by $B_5(26k+4s+5r)$.

\begin{center}
\begin{minipage}{4.9in}
\setlength{\unitlength}{0.0008cm}
\renewcommand{\dashlinestretch}{30}
\begin{picture}(7500, 5000)(0000, 0)
\thicklines
%
\path(0000,0000)(450,0000)
\path(0000,0300)(450,0300)
\path(0000,0600)(450,0600)
\path(0000,0900)(450,0900)
\path(0000,1200)(450,1200)
\path(0000,1500)(825,1500)
\path(0000,1800)(1500,1800)
\path(0000,2100)(2325,2100)
\path(0000,2400)(3000,2400)
\path(0000,2700)(3000,2700)
\path(0000,3000)(5325,3000)
\path(0000,3300)(6000,3300)
\path(0000,3600)(6000,3600)
\path(0000,3900)(6825,3900)
\path(0000,4200)(7500,4200)
\path(0000,4500)(7500,4500)
%
%
\path(7500,4200)(7500,4500)
\path(6825,3900)(6825,4200)
\path(6000,3300)(6000,4500)
\path(5325,3000)(5325,3300)
\path(4500,3000)(4500,4500)
\path(3000,2400)(3000,4500)
\path(2325,2100)(2325,2400)
\path(1500,1800)(1500,4500)
\path(825,1500)(825,1800)
\path(0450,1500)(0450,0000)
\path(0000,4500)(0000,0000)
\put(1800,2150){$s$} \put(0300,1550){$s$}
\put(0520,1000){$\vector(3,2){1450}$}
\put(0520,700){$\vector(3,2){3250}$}
\put(1190,400){$\vector(3,2){3800}$}
\put(520,400){$\vector(2,0){650}$}
\put(1390,100){$\vector(3,2){4890}$}
\put(6300,3350){$\vector(0,1){500}$}
\put(520,100){$\vector(2,0){850}$} \put(0150,1250){$r$}
\put(0150,950){$r$} \put(0150,650){$r$} \put(0150,0350){$r$}
\put(0150,0050){$r$}
\put(0600,1800){$k$}
\put(0600,2100){$k$}
\put(0600,2400){$k$}
\put(0600,2700){$k$}
\put(0600,3000){$k$}
\put(0600,3300){$k$}
\put(0600,3600){$k$}
\put(0600,3900){$k$}
\put(0600,4200){$k$}
\put(2100,2400){$k$}
\put(2100,2700){$k$}
\put(2100,3000){$k$}
\put(2100,3300){$k$}
\put(2100,3600){$k$}
\put(2100,3900){$k$}
\put(2100,4200){$k$}
\put(3600,3000){$k$}
\put(3600,3300){$k$}
\put(3600,3600){$k$}
\put(3600,3900){$k$}
\put(3600,4200){$k$}
\put(4800,3050){$s$}
\put(5100,3300){$k$}
\put(5100,3600){$k$}
\put(5100,3900){$k$}
\put(5100,4200){$k$}
\put(6300,3950){$s$}
\put(6600,4200){$k$}
%
%
%
\end{picture}
\begin{ca}
\label{partition1}
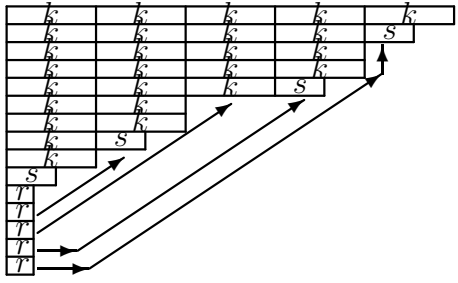
Place one part of size $r$ at the bottom of each of the 5 columns of width $k$.
\end{ca}
\end{minipage}
\end{center}

\begin{center}
\begin{minipage}{4.9in}
\setlength{\unitlength}{0.0008cm}
\renewcommand{\dashlinestretch}{30}
\begin{picture}(7500, 3600)(0000, 0)
\thicklines
%
\path(0000,000)(450,000)
\path(0000,300)(825,300)
\path(0000,600)(1950,600)
\path(0000,900)(2325,900)
\path(0000,1200)(3000,1200)
\path(0000,1500)(3450,1500)
\path(4500,1500)(4950,1500)
\path(0000,1800)(5325,1800)
\path(0000,2100)(6000,2100)
\path(0000,2400)(6450,2400)
\path(0000,2700)(6825,2700)
\path(0000,3000)(7500,3000)
\path(0000,3300)(7500,3300)
%
%
\path(7500,3000)(7500,3300)
\path(6825,2700)(6825,3000)
\path(6450,2400)(6450,2700)
\path(6000,2100)(6000,3300)
\path(5325,1800)(5325,2100)
\path(4500,1800)(4500,3300)
\path(3450,1500)(3450,1800)
\path(4500,1500)(4500,1800)
\path(4950,1500)(4950,1800)
\path(3000,1200)(3000,3300)
\path(2325,900)(2325,1200)
\path(1950,600)(1950,900)
\path(1500,600)(1500,3300)
\path(825,300)(825,600)
\path(0450,300)(0450,000)
\path(0000,3300)(0000,000)
\put(1800,950){$s$}
\put(0300,350){$s$}
\put(0150,050){$r$}
\put(1650,650){$r$}
\put(3150,1550){$r$}
\put(4650,1550){$r$}
\put(6150,2450){$r$}
\put(0600,600){$k$}
\put(0600,900){$k$}
\put(0600,1200){$k$}
\put(0600,1500){$k$}
\put(0600,1800){$k$}
\put(0600,2100){$k$}
\put(0600,2400){$k$}
\put(0600,2700){$k$}
\put(0600,3000){$k$}
\put(2100,1200){$k$}
\put(2100,1500){$k$}
\put(2100,1800){$k$}
\put(2100,2100){$k$}
\put(2100,2400){$k$}
\put(2100,2700){$k$}
\put(2100,3000){$k$}
\put(3600,1800){$k$}
\put(3600,2100){$k$}
\put(3600,2400){$k$}
\put(3600,2700){$k$}
\put(3600,3000){$k$}
\put(4800,1850){$s$}
\put(5100,2100){$k$}
\put(5100,2400){$k$}
\put(5100,2700){$k$}
\put(5100,3000){$k$}
\put(6300,2750){$s$}
\put(6600,3000){$k$}
%
%
%
\end{picture}
\begin{ca}
\label{partition2a}
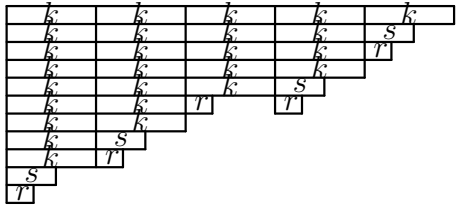
Now form the $k$-block conjugate of this partition.
\end{ca}
\end{minipage}
\end{center}

\begin{center}
\begin{minipage}{4.9in}
\setlength{\unitlength}{0.0008cm}
\renewcommand{\dashlinestretch}{30}
\begin{picture}(7500, 2000)(0000, 0)
\thicklines
%
\path(0000,000)(2775,000)
\path(0000,300)(7275,300)
\path(0000,600)(7950,600)
\path(0000,900)(11775,900)
\path(0000,1200)(14775,1200)
\path(0000,1500)(14775,1500)

%
%
\path(14775,1500)(14775,1200)
\path(13500,1500)(13500,1200)
\path(12000,1500)(12000,1200)
\path(11775,1200)(11775,900)
\path(10500,1500)(10500,900)
\path(9000,1500)(9000,900)
\path(7950,900)(7950,600)
\path(7500,1500)(7500,600)
\path(7275,600)(7275,300)
\path(6000,1500)(6000,300)
\path(4500,1500)(4500,300)
\path(3000,1500)(3000,300)
\path(2775,00)(2775,300)
\path(1500,1500)(1500,000)
\path(0000,1500)(0000,000)
\put(600,00){$k$}
\put(600,300){$k$}
\put(600,600){$k$}
\put(600,900){$k$}
\put(600,1200){$k$}

\put(1700,50){$r+s$}
\put(2100,300){$k$}
\put(2100,600){$k$}
\put(2100,900){$k$}
\put(2100,1200){$k$}

\put(3600,300){$k$}
\put(3600,600){$k$}
\put(3600,900){$k$}
\put(3600,1200){$k$}

\put(5100,300){$k$}
\put(5100,600){$k$}
\put(5100,900){$k$}
\put(5100,1200){$k$}

\put(6200,350){$r+s$}
\put(6600,600){$k$}
\put(6600,900){$k$}
\put(6600,1200){$k$}

\put(7650,650){$r$}
\put(8100,900){$k$}
\put(8100,1200){$k$}

\put(9600,900){$k$}
\put(9600,1200){$k$}

\put(10700,950){$r+s$}
\put(11100,1200){$k$}

\put(12600,1200){$k$}
\put(13700,1250){$r+s$}

\end{picture}
\begin{ca}
\label{partition3a}
This is a partition of the type counted by $B_5(26k+4s+5r)$.
\end{ca}
\end{minipage}
\end{center}
These steps are easily seen to be reversible.

\begin{lemma}
Let $k\geq 4$ be a fixed integer and let $r$ and $s$ be fixed integers such that $0<r<s<r+s<k$. Then
\begin{equation}\label{rskeq}
\sum_{n=0}^{\infty} \frac{(-q^s;q^k)_n\left(q^{r+k}\right)^n}{(q^k;q^k)_n}
=\frac{(-q^{s+r+k};q^k)_{\infty}}{(q^{r+k};q^k)_{\infty}}.
\end{equation}
\end{lemma}
\begin{proof}
The generating function for the sequence $A_n(m)$ is given by
\[
\frac{(-q^s;q^k)_n\left(q^{r+k}\right)^n}{(q^k;q^k)_n}=\sum_{m\geq (r+k)n}A_n(m)q^m.
\]
Thus
{\allowdisplaybreaks\begin{align*}
1+\sum_{n=1}^{\infty} \frac{(-q^s;q^k)_n\left(q^{r+k}\right)^n}{(q^k;q^k)_n}
&=1+\sum_{n=1}^{\infty}\sum_{m\geq (r+k)n}A_n(m)q^m\\
&=1+\sum_{n=1}^{\infty}\sum_{m\geq (r+k)n}B_n(m)q^m\\
&=1+\sum_{m\geq (r+k)}B(m)q^m,
\end{align*}}
where $B(m)$ counts the number of partitions of $m$ with
\begin{itemize}
  \item distinct parts $\equiv r+s (\mod k)$, with the part $r+s$ not appearing,
 \item possibly repeating parts $\equiv r (\mod k)$, with the part $r$ not appearing.
  \end{itemize}
It is  clear that
\[
1+\sum_{m\geq (r+k)}B(m)q^m=\frac{(-q^{s+r+k};q^k)_{\infty}}{(q^{r+k};q^k)_{\infty}},
\]
and the result now follows.
\end{proof}

We now give a proof of the $q$-Binomial Theorem.
\begin{theorem}
Let $a$, $z$ and $q$ be complex numbers with $|z|, |q|<1$. Then
\begin{equation}\label{bintheoeq}
\sum_{n=0}^{\infty}\frac{(a;q)_nz^n}{(q;q)_n}
=\frac{(az;q)_{\infty}}{(z;q)_{\infty}}.
\end{equation}
\end{theorem}

\begin{proof}
By \eqref{rskeq}, if $k$ and $m$ are positive integers with $k\geq 4$, and $r$ and $s$ are integers with $0<r<sm<sm+r<mk$, then
\begin{equation*}
\sum_{n=0}^{\infty} \frac{(-q^{sm};q^{km})_n\left(q^{r+km}\right)^n}{(q^{km};q^{km})_n}
=\frac{(-q^{sm+r+km};q^{km})_{\infty}}{(q^{r+km};q^{km})_{\infty}}.
\end{equation*}
Fix an $m$-th root of $q$, denoted $q^{1/m}$, and replace $q$ with $q^{1/m}$ to get
\begin{equation*}
\sum_{n=0}^{\infty} \frac{(-q^{s};q^{k})_n\left(q^{r/m+k}\right)^n}{(q^{k};q^{k})_n}
=\frac{(-q^{s+r/m+k};q^{k})_{\infty}}{(q^{r/m+k};q^{k})_{\infty}}.
\end{equation*}
Now let $m$ take the values $1,2,3,\dots$, so that the identity
\begin{equation}\label{kszeq}
\sum_{n=0}^{\infty} \frac{(-q^{s};q^{k})_n\left(zq^{k}\right)^n}{(q^{k};q^{k})_n}
=\frac{(-q^{s+k}z;q^{k})_{\infty}}{(zq^{k};q^{k})_{\infty}}
\end{equation}
holds for $z\in \{q^{r/m}:m\geq 1\}$. By continuity this identity also holds for $z=1$, the limit of this sequence. Hence, by the Identity Theorem, \eqref{kszeq} holds for $|z|<|q|^{-k}$. Replace $z$ with $z/q^k$ and we get that
\begin{equation}\label{kszeq2}
\sum_{n=0}^{\infty} \frac{(-q^{s};q^{k})_nz^n}{(q^{k};q^{k})_n}
=\frac{(-q^{s}z;q^{k})_{\infty}}{(z;q^{k})_{\infty}}
\end{equation}
holds for $|z|<1$ and $1<s<k$.

Next, fix a $k$-th root of $q$, denoted $q^{1/k}$, replace $q$ with $q^{1/k}$ in \eqref{kszeq2} to get that
\begin{equation}\label{szeq}
\sum_{n=0}^{\infty} \frac{(-q^{s/k};q)_nz^n}{(q;q)_n}
=\frac{(-q^{s/k}z;q)_{\infty}}{(z;q)_{\infty}}.
\end{equation}
Set $s=2$ and let $k$ take the values $4, 5, 6, \dots$ to get that
\begin{equation}\label{azeq}
\sum_{n=0}^{\infty} \frac{(a;q)_nz^n}{(q;q)_n}
=\frac{(az;q)_{\infty}}{(z;q)_{\infty}}
\end{equation}
holds for $a\in \{-q^{2/k}:k\geq 4\}$ and $|z|<1$. By continuity, \eqref{azeq} also holds for $a=-1$, the limit point of this sequence. Thus, again by the Identity Theorem, \eqref{azeq} holds for all $a \in \mathbb{C}$ and all $z\in \mathbb{C}$ with $|z|<1$.
\end{proof}

\section{Some Preliminary Summation Formulae}
Before coming to the proof of the next identities, we prove some
preliminary lemmas.

\begin{lemma}\label{l1}
Let $|q|<1$ and $b \not=-q^{-n}$ for any positive integer $n$. Then if $m$ is any positive integer,
\begin{equation}\label{lid2}
\sum_{0\leq a_1\leq a_2\leq \dots \leq a_n}\frac{q^{m(a_1+a_2+\dots +
a_n)}} {\prod_{j=0}^{n-1}\prod_{k=1}^{m+1}(1+bq^{j(m+1)+k+a_{j+1}})}
=\frac{1}{(q^m;q^m)_n(-bq;q)_{mn}},
\end{equation}
where the sum is over all $n$-tuples $\{a_1,\dots, a_n\}$  of integers that satisfy the stated inequality.
\end{lemma}

\begin{proof}
We rewrite the left side of \eqref{lid2} as the nested sum
\begin{multline}\label{nestedsum}
\sum_{a_1\geq0}\frac{q^{ma_1}}{\prod_{k=1}^{m+1}(1+b q^{k+a_1})}
\sum_{a_2\geq a_1}\frac{q^{ma_2}}{\prod_{k=1}^{m+1}(1+b q^{(m+1)+k+a_2})}\\
\dots
\sum_{a_{n-1}\geq a_{n-2}}\frac{q^{ma_{n-1}}}{\prod_{k=1}^{m+1}(1+b q^{(n-2)(m+1)+k+a_{n-1}})}\\
\sum_{a_{n}\geq a_{n-1}}\frac{q^{ma_{n}}}{\prod_{k=1}^{m+1}(1+b q^{(n-1)(m+1)+k+a_n})}
\end{multline}

Next, we note that if $p \geq 1$ is an integer, and none of the denominators following vanish, that
\begin{multline}\label{telesum}
\sum_{a_{i}\geq a_{i-1}}\frac{q^{pma_{i}}}{\prod_{k=1}^{mp+1}(1+c q^{k+a_i})}\\=
\frac{1}{1-q^{pm}}\sum_{a_{i}\geq a_{i-1}}\bigg [\frac{q^{pma_{i}}}{\prod_{k=1}^{mp}(1+c q^{k+a_i})}-\frac{q^{pm(a_{i}+1)}}{\prod_{k=2}^{mp+1}(1+c q^{k+a_i})}\bigg ]\\
=\frac{1}{1-q^{pm}}\frac{q^{pma_{i-1}}}{\prod_{k=1}^{mp}(1+c q^{k+a_{i-1}})},
\end{multline}
since the second sum telescopes. We now apply this result (with $p=1$) to the innermost sum at \eqref{nestedsum} to get that this sum has the value
\[
\frac{q^{ma_{n-1}}}{(1-q^m)\prod_{k=1}^{m}(1+ bq^{(n-1)(m+1)+k+a_{n-1}})},
\]
so that the next innermost sum at \eqref{nestedsum} becomes
\[
\sum_{a_{n-1}\geq
a_{n-2}}\frac{q^{2ma_{n-1}}}{(1-q^m)\prod_{k=1}^{2m+1}(1+
bq^{(n-2)(m+1)+k+a_{n-1}})}.
\]
We apply \eqref{telesum} again, this time with $p=2$, to get that this sum has value
\[
\frac{q^{2ma_{n-2}}}{(1-q^m)(1-q^{2m})\prod_{k=1}^{2m}(1+ bq^{(n-2)(m+1)+k+a_{n-2}})}.
\]
This now results in the third innermost sum becomes
\begin{equation*}
\sum_{a_{n-2}\geq a_{n-3}}\frac{1}{(q^m;q^m)_2}
\frac{q^{3ma_{n-2}}}{\prod_{k=1}^{3m+1}(1+
bq^{(n-3)(m+1)+k+a_{n-2}})}.
\end{equation*}
This process can be continued, so that after $n-1$ steps, the left side of \eqref{nestedsum} equals
\begin{multline}\label{lilastsum}
\sum_{a_1\geq 0}\frac{q^{mna_1}}{(q^m;q^m)_{n-1}\prod_{k=1}^{nm+1}(1+ bq^{k+a_{1}})}\\
= \frac{q^{mn (0)}}{(q^m;q^m)_{n-1}(1-q^{nm})\prod_{k=1}^{nm}(1+
bq^{k+0})} =\frac{1}{(q^m;q^m)_n(-bq;q)_{nm}},
\end{multline}
giving the result.
\end{proof}

\begin{lemma}\label{l2}
Let $|q|<1$ and $b \not=-q^{-n}$ for any positive integer $n$. Then
if $m$ is any positive integer,
\begin{equation}\label{lid3}
\sum_{0\leq a_1\leq a_2\leq \dots \leq
a_n}^{'}\frac{q^{m(a_1+a_2+\dots +
a_n)}}{\prod_{j=0}^{n-1}\prod_{k=1}^{m+1}(1+bq^{jm+k+a_{j+1}})}
=\frac{1}{(q^m;q^m)_n(-bq;q)_{mn}},
\end{equation}
where the sum is over all $n$-tuples $\{a_1,\dots, a_n\}$  of
integers that satisfy the stated inequality, and the $\sum^{'}$
notation means that if $a_i=a_{i-1}$ for any $i$, then the factor
$1+b q^{(i-1)m+m+1+a_{i-1}}=1+b q^{im+1+a_i}$ occurs just once in
any product.
\end{lemma}

\begin{proof}
The proof is similar to the proof of Lemma \ref{l1}. We rewrite the
left side of \eqref{lid3} as the nested sum
\begin{multline}\label{nestedsum2}
\sum_{a_1\geq0}\frac{q^{ma_1}}{\prod_{k=1}^{m+1}(1+b q^{k+a_1})}
\sum_{a_2\geq a_1}^{'}\frac{q^{ma_2}}{\prod_{k=1}^{m+1}(1+b
q^{m+k+a_2})}\\ \dots \sum_{a_{n-1}\geq
a_{n-2}}^{'}\frac{q^{ma_{n-1}}}{\prod_{k=1}^{m+1}(1+b
q^{(n-2)m+k+a_{n-1}})}
\\
\sum_{a_{n}\geq a_{n-1}}^{'}\frac{q^{ma_{n}}}{\prod_{k=1}^{m+1}(1+b
q^{(n-1)m+k+a_n})}.
\end{multline}

Next, we note that if $p \geq 1$ is an integer, and the term $1+c
q^{1+a_{i-1}}$ occurs in the next sum out, and none of the
denominators following vanish, then
\begin{multline}\label{telesum2} \sum_{a_{i}\geq
a_{i-1}}^{'}\frac{q^{pma_{i}}}{\prod_{k=1}^{mp+1}(1+c q^{k+a_i})}\\=
\frac{q^{pma_{i-1}}}{\prod_{k=2}^{mp+1}(1+c
q^{k+a_{i-1}})}+\sum_{a_{i}\geq
a_{i-1}+1}\frac{q^{pma_{i}}}{\prod_{k=1}^{mp+1}(1+c q^{k+a_i})}\\
=\frac{q^{pma_{i-1}}}{\prod_{k=2}^{mp+1}(1+c q^{k+a_{i-1}})}+
\frac{1}{1-q^{pm}}\frac{q^{pm(a_{i}+1)}}{\prod_{k=1}^{mp}(1+c
q^{k+a_{i-1}+1})}\\
=\frac{q^{pma_{i-1}}}{(1-q^{pm})\prod_{k=2}^{mp+1}(1+c
q^{k+a_{i-1}})},
\end{multline}
where the second equality follows from the same telescoping argument
used in Lemma \ref{l1}. We now apply this summation result
repeatedly, starting  with the innermost sum at \eqref{nestedsum2}
(with (with $p=1$)), to eventually arrive at the sum at
\eqref{lilastsum} above, thus giving the result.
\end{proof}

\begin{lemma}\label{l3}
Let $|q|<1$ and $b \not=-q^{-n}$ for any positive integer $n$. Then
if $m$ is any positive integer,
\begin{equation}\label{lid4}
\sum_{0\leq a_1\leq a_2\leq \dots \leq
a_n}^{''}\frac{q^{m(a_1+a_2+\dots +
a_n)}}{\prod_{j=0}^{n-1}\prod_{k=0}^{m}(1+bq^{jm+k+a_{j+1}})}
=\frac{1}{(q^m;q^m)_n(-bq;q)_{mn}},
\end{equation}
where the sum is over all $n$-tuples $\{a_1,\dots, a_n\}$  of
integers that satisfy the stated inequality, and the $\sum^{''}$
notation means that if $a_i=a_{i-1}$ for any $i$, then the factor
$1+b q^{(i-1)m+m+a_{i-1}}=1+b q^{im+0+a_i}$ occurs just once in any
denominator product, and in addition, if $a_1=0$, then the factor
$1+b=1+b q^{0+0}$ does not appear in any denominator product.
\end{lemma}

\begin{proof}
The proof parallels the proof of Lemma \ref{l2}, to get after $n-1$
steps, that the left side of \eqref{lid4} equals
\begin{multline}\label{lilastsuml3}
\sum_{a_1\geq
0}^{''}\frac{q^{mna_1}}{(q^m;q^m)_{n-1}\prod_{k=0}^{nm}(1+
bq^{k+a_{1}})}\\ = \frac{q^{mn
(0)}}{(q^m;q^m)_{n-1}\prod_{k=1}^{nm}(1+ bq^{k+0})} + \sum_{a_1\geq
1}\frac{q^{mna_1}}{(q^m;q^m)_{n-1}\prod_{k=0}^{nm}(1+
bq^{k+a_{1}})}\\
= \frac{1}{(q^m;q^m)_{n-1}(-bq;q)_{mn}}+ \frac{q^{mn
(1)}}{(q^m;q^m)_{n-1}(1-q^{mn})\prod_{k=0}^{nm-1}(1+ bq^{k+1})}\\
 =\frac{1}{(q^m;q^m)_n(-bq;q)_{nm}}.
\end{multline}
\end{proof}

\section{Hybrid proofs of some $q$-series Identities}

We recall the second iterate of Heine's transformation (see \cite[page 38]{A76}).
\begin{equation}\label{heine2}
\sum_{n=0}^{\infty}\frac{(a,b;q)_n}{(c,q;q)_n}t^n
=\frac{(c/b,bt;q)_{\infty}}{(c,t;q)_{\infty}}
\sum_{n=0}^{\infty}\frac{(abt/c,b;q)_n}{(bt,q;q)_n}\left(\frac{c}{b}\right)^n.
\end{equation}
We will give a hybrid proof of a special case (set $c=0$, replace $a$ with $-a$ and $b$ with $-bq/t$, and finally let $t \to 0$) of this
identity.

\begin{theorem}\label{t2}
\begin{equation}\label{hh}
\sum_{n=0}^{\infty}\frac{(-a;q)_nb^nq^{n(n+1)/2}}{(q;q)_n}
=(-bq;q)_{\infty}
\sum_{n=0}^{\infty}\frac{(ab)^nq^{n^2}}{(q,-bq;q)_n}.
\end{equation}
\end{theorem}

Remark:  A version of \eqref{hh} was stated by Ramanujan, see for example \cite[\textbf{Entry 1.6.1, page 24}]{AB09}. Proofs of \eqref{hh} have been given by Ramamani \cite{R70} and
Ramamani and  Venkatachaliengar \cite{RV72}. A generalization of \eqref{hh} was proved by Bhargava and  Adiga \cite{BA86}, while Srivastava \cite{S87} showed that \eqref{hh} follows as a special case of Heine's transformation, as described above. Lastly, a combinatorial proof of \eqref{hh} has been given in \cite{BKY10}
by Berndt, Kim and Yee.

\begin{proof}[Proof of Theorem \ref{t2}]
We will prove  for all integers $r$, $s$ and $k$ satisfying $0<r<s<r+s<k$, that
\begin{equation}\label{hhk}
\sum_{n=0}^{\infty}\frac{(-q^s;q)_nq^{rn}q^{kn(n+1)/2}}{(q^k;q^k)_n}
=(-q^{r+k};q^k)_{\infty}
\sum_{n=0}^{\infty}\frac{q^{(s+r)n}q^{kn^2}}{(q^k,-q^{r+k};q^k)_n},
\end{equation}
and \eqref{hh} will then follow from the Identity Theorem, by an argument similar to that used in the proof of the
$q$-Binomial Theorem.

The $n$-th term in the series on the left side of \eqref{hhk} may be regarded as the generating function for partitions with
\begin{itemize}
\item the part $r$ occurring exactly $n$ times,
 \item distinct parts from  $\{s,s+k,s+2k,\dots, s+(n-1)k\}$,
 \item possibly repeating parts from $\{k,2k,3k,\dots ,nk\}$, with each part  occurring at least once.
  \end{itemize}

We consider the Ferrers diagram for such a partition, which may be regarded as having $n$ columns, each of width $k$. We first distribute the $n$ parts of size $r$ so that one such part is placed at the bottom of each column.
We then take the $k$-block conjugate of this partition we get a
partition into $n$ parts with
\begin{itemize}
 \item distinct parts  $\equiv s+r(\mod k)$, with the part $s+r$ not
 appearing and a gap of at least $2k$ between consecutive parts,
 \item distinct parts  $\equiv r (\mod k)$, with the parts $r+jk$ and $r+(j+1)k$ not appearing if the part $r+s+jk$ appears
 (here $j \geq 1$).
  \end{itemize}

Once again, this operation of taking the $k$-block conjugate gives a
bijection between these two sets of partitions. If we now sum over
all $n$, we get all partitions with
\begin{itemize}
 \item distinct parts  $\equiv s+r(\mod k)$, with the part $s+r$ not
 appearing and a gap of at least $2k$ between consecutive parts,
 \item distinct parts  $\equiv r (\mod k)$, with the parts $r+jk$ and $r+(j+1)k$ not appearing if the part $r+s+jk$ appears
 (here $j \geq 1$).
  \end{itemize}

Next, instead of considering partitions of this latter type where there are a total of $n$ parts, we consider instead partitions of this type containing exactly $n$ parts $\equiv r+s(\mod k)$. In other words we consider partitions with
\begin{itemize}
 \item exactly $n$ distinct parts  $\equiv s+r(\mod k)$, with the part $s+r$ not
 appearing and a gap of at least $2k$ between consecutive parts,
 \item distinct parts  $\equiv r (\mod k)$, with the parts $r+jk$ and $r+(j+1)k$ not appearing if the part $r+s+jk$ appears
 (here $j \geq 1$).
  \end{itemize}
It is not difficult to see that the generating function for such partitions is
\begin{multline*}
\sum_{0\leq a_1\leq \dots \leq a_n} \frac{q^{(r+s+(1+a_1)k)+(r+s+(3+a_2)k)+\dots + (r+s+(2n-1+a_n)k)}(-q^{r+k};q^k)_{\infty}}
{\prod_{j=1}^{n}(1+q^{r+(2j-1+a_j)k})(1+q^{r+(2j+a_j)k})}\\
=(-q^{r+k};q^k)_{\infty}q^{(r+s) n}q^{kn^2}
\phantom{sdasdasdasdasdas}\\
\times \sum_{0\leq a_1\leq a_2\leq \dots \leq a_n} \frac{q^{(a_1+a_2+\dots + a_n)k}}
{\prod_{j=1}^{n}(1+q^{r+(2j-1+a_j)k})(1+q^{r+(2j+a_j)k})}\\
=(-q^{r+k};q^k)_{\infty}\frac{q^{(r+s) n}q^{kn^2}}{(q^{k};q^{k})_n(-q^{r+k};q^k)_n},
\end{multline*}
where the last equality follows from \eqref{lid2} (with $m=1$,
$b=q^r$ and $q$ replaced with $q^k$). Now summing over all $n$ gives
\eqref{hhk}, and \eqref{hh} follows.
\end{proof}

We now prove  a pair of identities stated by Ramanujan
(\cite[\textbf{Entry 1.5.1}, page 23]{AB09}, $a$ replaced with
$aq$). Analytic proofs were given by Watson \cite{W37} and Andrews
\cite{A66}, and a combinatorial proof has been given in \cite{BKY10}
by Berndt, Kim and Yee.

\begin{theorem}\label{t3}
If $|q|<1$ and $a \not = -q^{-2n}$ for any integer $n>0$, then
{\allowdisplaybreaks\begin{align}\label{151eq}
\sum_{n=0}^{\infty}\frac{a^nq^{n^2+n}}{(q;q)_n}&=(-aq^2;q^2)_{\infty}
\sum_{n=0}^{\infty}\frac{a^nq^{n^2+2n}}{(q^2;q^2)_n(-aq^2;q^2)_n},\\
&=(-aq^3;q^2)_{\infty}
\sum_{n=0}^{\infty}\frac{a^nq^{n^2+n}}{(q^2;q^2)_n(-aq^3;q^2)_n}. \notag
\end{align}}
\end{theorem}
\begin{proof}
We will prove only the first identity, as the proof of the second is very similar.
We will first show,  for all integers $0<r<k$, that
\begin{equation}\label{151eqk}
\sum_{n=0}^{\infty}\frac{q^{rn}q^{k(n^2+n)}}{(q^k;q^k)_n}
=(-q^{r+2k};q^{2k})_{\infty}
\sum_{n=0}^{\infty}
\frac{q^{rn}q^{k(n^2+2n)}}{(q^{2k};q^{2k})_n(-q^{r+2k};q^{2k})_n}.
\end{equation}

The $n$-th term in the series on the left side of \eqref{151eqk} may be interpreted as the generating function for partitions with
\begin{itemize}
\item the part $r$ occurring exactly $n$ times,
 \item  repeating parts from $\{k,2k,3k,\dots ,nk\}$, with each part  occurring at least twice.
  \end{itemize}

We once again consider the Ferrers diagram for such a partition, which also may be regarded as having $n$ columns, each of width $k$. We first distribute the $n$ parts of size $r$ so that one such part is placed at the bottom of each column.
We then take the $k$-block conjugate of this partition we get a
partition into $n$ parts with
\begin{itemize}
 \item distinct parts  $\equiv r(\mod k)$, with the parts $r$ and $r+k$ not
 appearing and a gap of at least $2k$ between consecutive parts.
  \end{itemize}

If we now sum over
all $n$, we get all partitions with
\begin{itemize}
 \item distinct parts  $\equiv r(\mod k)$, with the parts $r$ and $r+k$ not
 appearing and a gap of at least $2k$ between consecutive parts.
  \end{itemize}

We consider instead partitions of this type containing exactly $n$ distinct parts $\equiv r+k(\mod 2k)$, with the part $r+k$ not appearing, and distinct parts $\equiv r(\mod 2k)$, with the part $r$ not appearing and a  gap of at least $2k$  between any consecutive parts. (If there are no  parts $\equiv r+k(\mod 2k)$, then the partition consists entirely of distinct parts $\equiv r(\mod 2k)$, with the part $r$ not appearing, and these partitions have generating function $(-q^{r+2k};q^{2k})_{\infty}$). In other words we consider partitions with
\begin{itemize}
 \item exactly $n$ distinct parts  $\equiv r+k(\mod 2k)$, with the part $r+k$ not
 appearing and a gap of at least $2k$ between consecutive parts,
 \item distinct parts  $\equiv r (\mod 2k)$, with the part $r$ not appearing, and with the parts $r+(2j-2)k$ and $r+2jk$ not appearing if the part $r+(2j-1)k$ appears
 (here $j \geq 2$).
  \end{itemize}
The generating function for such partitions is
\begin{multline*}
\sum_{0\leq a_1\leq \dots \leq a_n}' \frac{q^{(r+(3+2a_1)k)+(r+(5+2a_2)k)+\dots + (r+(2n+1+2a_n)k)}(-q^{r+2k};q^{2k})_{\infty}}
{\prod_{j=2}^{n+1}(1+q^{r+(2j-2+2a_j)k})(1+q^{r+(2j+2a_j)k})}\\
=(-q^{r+2k};q^{2k})_{\infty}q^{r n}q^{k(n^2+2n)}
\phantom{sdasdasdasdasdas}\\
\times \sum_{0\leq a_1\leq a_2\leq \dots \leq a_n}' \frac{q^{(a_1+a_2+\dots + a_n)2k}}
{\prod_{j=1}^{n}(1+q^{r+(j+a_j)2k})(1+q^{r+(j+1+a_j)2k})}\\
=(-q^{r+2k};q^{2k})_{\infty}\frac{q^{r n}q^{k(n^2+2n)}}{(q^{2k};q^{2k})_n(-q^{r+2k};q^{2k})_n},
\end{multline*}
where the last equality follows from \eqref{lid3} (with $b=q^r$,
$m=1$ and $q$ replaced with $q^{2k}$). Now summing over all $n$
gives \eqref{151eqk}, and the first identity at\eqref{151eq} follows
once again by the Identity Theorem.

The proof of the second identity is similar, except that instead of considering partitions with exactly $n$ parts $\equiv r+k (\mod 2k)$ with the part $r+k$ not appearing, we consider partitions with exactly $n$ parts $\equiv r (\mod 2k)$ with the part $r$ not appearing. The second identity at \eqref{151eq} then follows, after some minor technicalities.
\end{proof}

Next we give a hybrid proof of a special case of another identity of
Ramanujan (see \textbf{Entry 1.4.17} on page 22 of \cite{AB09}).
\begin{theorem}\label{t4} If $|q|<1$ and $a,b \not = -q^{-n}$ for any positive integer $n$, then
\begin{equation}\label{symeq}
(-bq;q)_{\infty} \sum_{n=0}^{\infty}
\frac{a^nq^{n(n+1)/2}}{(q;q)_n(-bq;q)_n} = (-aq;q)_{\infty}
\sum_{n=0}^{\infty} \frac{b^nq^{n(n+1)/2}}{(q;q)_n(-aq;q)_n}.
\end{equation}
\end{theorem}
Remark: In the more general identity stated by Ramanujan, the terms
$(-aq;q)_n$ and $(-bq;q)_n$ above are replaced, respectively, with
$(-aq;q)_{mn}$ and $(-bq;q)_{mn}$, where $m$ is any positive
integer. A combinatorial proof of Ramanujan's identity has been
given in \cite{BKY10} by Berndt, Kim and Yee.
\begin{proof}
We will show for all integers $r$, $s$, $k$ satisfying $0<r<s<k$ that
\begin{equation}\label{symeqk}
(-q^{r+k};q^k)_{\infty} \sum_{n=0}^{\infty}
\frac{q^{s n}q^{k n(n+1)/2}}{(q^k,-q^{r+k};q^k)_n} =
(-q^{s+k};q^k)_{\infty} \sum_{n=0}^{\infty}
\frac{q^{r n}q^{kn(n+1)/2}}{(q^k,-q^{s+k};q^k)_n},
\end{equation}
and the full result at \eqref{symeq} will follow once again from the Identity Theorem.

By \eqref{lid3} (with $m=1$ and $q$ replaced with $q^k$), the left side of \eqref{symeqk} equals
\begin{multline*}
\sum_{n=0}^{\infty}\sum_{0\leq a_1\leq a_2\leq \dots \leq
a_n}^{'}\frac{  q^{\sum_{j=1}^n s+(j+a_j)k}(-q^{r+k};q^k)_{\infty}}
{\prod_{j=0}^{n-1}
(1+q^{r+(j+1+a_{j+1})k})(1+q^{r+(j+2+a_{j+1})k})}\\
=\sum_{n=0}^{\infty}\sum_{0\leq a_1\leq a_2\leq \dots \leq
a_n}^{'}\frac{  q^{\sum_{j=1}^n s+(j+a_j)k}(-q^{r+k};q^k)_{\infty}}
{\prod_{j=1}^{n}
(1+q^{r+(j+a_{j})k})(1+q^{r+(j+1+a_{j})k})}.
\end{multline*}
The $n$-th term of this latter series may be regarded as the generating function for partitions with
\begin{itemize}
 \item exactly $n$ distinct parts  $\equiv s(\mod k)$, with the part $s$ not
 appearing,
 \item distinct parts  $\equiv r (\mod k)$, with the part $r$ not appearing, and with the parts $r+pk$ and $r+(p+1)k$ not appearing if the part $s+pk$ appears
 (here $p \geq 1$),
  \end{itemize}
and so the entire series may be regarded as the generating function for partitions with
\begin{itemize}
 \item  distinct parts  $\equiv s(\mod k)$, with the part $s$ not
 appearing,
 \item distinct parts  $\equiv r (\mod k)$, with the part $r$ not appearing, and with the parts $r+pk$ and $r+(p+1)k$ not appearing if the part $s+pk$ appears
 (here $p \geq 1$).
  \end{itemize}
These conditions are equivalent to the conditions
\begin{itemize}
 \item  distinct parts  $\equiv r(\mod k)$, with the part $r$ not
 appearing,
 \item distinct parts  $\equiv s (\mod k)$, with the part $s$ not appearing, and with the parts $s+(p-1)k$ and $s+pk$ not appearing if the part $r+pk$ appears
 (here $p \geq 1$).
  \end{itemize}
The generating function for such partitions containing exactly $n$ distinct parts $\equiv r(\mod k)$ is
\begin{multline*}
\sum_{0\leq a_1\leq a_2\leq \dots \leq
a_n}^{''}\frac{  q^{\sum_{j=1}^n r+(j+a_j)k}(-q^{s+k};q^k)_{\infty}}
{\prod_{j=1}^{n}
(1+q^{s+(j-1+a_{j})k})(1+q^{s+(j+a_{j})k})}\\
=\sum_{0\leq a_1\leq a_2\leq \dots \leq
a_n}^{''}\frac{  q^{\sum_{j=1}^n r+(j+a_j)k}(-q^{s+k};q^k)_{\infty}}
{\prod_{j=0}^{n-1}
(1+q^{s+(j+a_{j+1})k})(1+q^{s+(j+1+a_{j+1})k})}\\
=\frac{(-q^{s+k};q^k)_{\infty}q^{rn}q^{kn(n+1)/2}}{(q^k,-q^{s+k};q^k)_n},
\end{multline*}
where the last equality follows from \eqref{lid4} (with $m=1$, $q$ replaced with $q^k$, and $b=q^s$). The identity at \eqref{symeqk} now follows upon summing over all $n$.
\end{proof}

\section{Some New Partitions Identities Deriving from Identities of Rogers-Ramanujan-Slater type}

Lemmas \ref{l1} - \ref{l3} allow us to derive new partition interpretations
of some well-known analytic identities.

\subsection{The Rogers-Ramanujan Identities}

The following identities appear in Slater's paper \cite{S52}
(\textbf{S14} refers to the identity numbered (14) in Slater's paper
\cite{S52}, and similarly for other identities labelled below):
\begin{align}\label{rr1}
\sum_{n=0}^{\infty}\frac{q^{n^2+n}}{(q;q)_n}&=\frac{1}{(q^2,q^3;q^5)_{\infty}},
\tag{\textbf{S14}}\\
(-q^2;q^2)_{\infty}\sum_{n=0}^{\infty}\frac{q^{n^2+2n}}{(q^4;q^4)_n}
&=\frac{1}{(q^2,q^3;q^5)_{\infty}},
\tag{\textbf{S16}}\\
(-q;q^2)_{\infty}\sum_{n=0}^{\infty}\frac{q^{n^2+n}}{(q^2;q^2)_n(-q;q^2)_{n+1}}
&=\frac{1}{(q^2,q^3;q^5)_{\infty}}. \tag{\textbf{S94}}
\end{align}

Each of these identities had also previously been proven by Rogers
\cite{R94}. The equality of the three left
sides of  these  equations easily follow from Theorem \ref{t3}, and in fact they could also be proved directly from the summation formulae in Lemmas
\ref{l2} and \ref{l3}.

Perhaps more interesting is the result of interpreting the left
sides of \textbf{S16} and \textbf{S94} using the summation formula
in Lemma \ref{l1}. As is well known, the identity at \textbf{S14} (\emph{The Second Rogers-Ramanujan Identity}) implies that if $A(n)$
denotes the number of partitions of $n$ into distinct parts with no
1's and a gap of at least 2 between consecutive parts, and $B(n)$
denotes the number of partitions of $n$ into parts $\equiv 2, 3
(\mod 5)$, then $A(n)=B(n)$ for all positive integers $n$. Lemma
\ref{l1} now lets us describe two other sets of partitions of each
positive integer $n$ which are also equinumerous with the sets of
partitions counted by $A(n)$ and $B(n)$.

\begin{theorem}
For a positive integer $n$, let $A(n)$ denotes the number of partitions of $n$ into distinct parts
with no 1's and a gap of at least 2 between consecutive parts, and let $B(n)$
denote the number of partitions of $n$ into parts $\equiv 2, 3 (\mod 5)$.

Let $C(n)$ denote the number of partitions of $n$ into distinct parts with no 1's appearing,
such that if $o_j$ is the j-th odd part (where we order the parts in ascending order),
then the even parts $o_j+2j-3$ and $o_j +2j-1$ do not appear.

Let $D(n)$ denote the number of partitions of $n$ into distinct parts with no 1's appearing,
such that if $e_j$ is the j-th even part (where again we order the parts in ascending order),
then the odd parts $e_j+2j-1$ and $e_j +2j+1$ do not appear.

Then
\begin{equation}
A(n)=B(n)=C(n)=D(n).
\end{equation}
\end{theorem}

\begin{proof}
From what has been said already about the equality of the three left sides at \textbf{S14}, \textbf{S16} and \textbf{S94}, all that is necessary is to show that
\begin{align*}
(-q^2;q^2)_{\infty}\sum_{n=0}^{\infty}\frac{q^{n^2+2n}}{(q^4;q^4)_n}
&=\sum_{n=0}^{\infty}C(n)q^n\\
(-q;q^2)_{\infty}\sum_{n=0}^{\infty}\frac{q^{n^2+n}}{(q^2;q^2)_n(-q;q^2)_{n+1}}
&=\sum_{n=0}^{\infty}D(n)q^n.
\end{align*}

We do this for the second identity only, since the proof for the former follows similarly. By Lemma \ref{l1}, with $q$ replaced with $q^2$, $m=1$ and $b=q$,
\begin{align*}
(-q;q^2)_{\infty}&\sum_{n=0}^{\infty}\frac{q^{n^2+n}}{(q^2;q^2)_n(-q;q^2)_{n+1}}
=(-q^3;q^2)_{\infty}\sum_{n=0}^{\infty}\frac{q^{n^2+n}}{(q^2;q^2)_n(-q^3;q^2)_{n}}\\
&=\sum_{n=0}^{\infty}\sum_{0\leq a_1\leq a_2\leq \dots \leq a_n}\frac{q^{(2+2a_1)+(4+2a_2)+\dots+
(2n+2a_n)}(-q^3;q^2)_{\infty} } {\prod_{j=1}^{n}(1+q^{4j-1+2a_{j}})(1+q^{4j+1+2a_{j}})}.
\end{align*}
This last series is the generating function for partitions into distinct parts, with no 1's appearing, and such that if $2j+2a_j$ is the $j$-th even part, then the odd parts $(2j+2a_j) +2j-1$ and $(2j+2a_j) +2j+1$ do not appear. This is precisely the partitions of an integer $n$ counted by $D(n)$.
\end{proof}

As an example we consider the nine partitions of 15 counted by $B(15)$, $C(15)$ and $D(15)$. Those counted by $B(15)$ are
\begin{multline*}
\{3, 2, 2, 2, 2, 2, 2\}, \{3, 3, 3, 2, 2, 2\}, \{3, 3, 3, 3, 3\},
\{7, 2, 2, 2, 2\}, \{7, 3, 3, 2\},\\ \{8, 3, 2, 2\}, \{8, 7\}, \{12, 3\}, \{13, 2\},
\end{multline*}
those counted by $C(15)$ are
{\allowdisplaybreaks
\begin{equation*}
\{7, 5, 3\}, \{8, 5, 2\},
 \{9, 4, 2\}, \{9, 6\}, \{10, 5\}, \{11, 4\}, \{12, 3\}, \{13, 2\}, \{15\},
\end{equation*}}
while those counted by $D(15)$ are
\begin{equation*}
\{7, 5, 3\}, \{7, 6, 2\},
 \{8, 4, 3\}, \{8, 7\}, \{10, 5\}, \{11, 4\}, \{12, 3\}, \{13, 2\}, \{15\}.
\end{equation*}

Note that $D(15)$ does not count, for example, $\{8, 5, 2\}$ (since $e_1=2$ and $e_1+2(1)+1=5$), while $C(15)$ does not count, for example, $\{7, 6, 2\}$ (since $o_1=7$ and $o_1+2(1)-3=6$).

Three partner identities which also appear in Slater's paper
\cite{S52} and which were also previously proven by Rogers
\cite{R94} are the following: {\allowdisplaybreaks\begin{align}
\sum_{n=0}^{\infty}\frac{q^{n^2}}{(q;q)_n}&=\frac{1}{(q,q^4;q^5)_{\infty}},
\tag{\textbf{S18}}\\
(-q^2;q^2)_{\infty}\sum_{n=0}^{\infty}\frac{q^{n^2}}{(q^2,-q^2;q^2)_n}
&=\frac{1}{(q,q^4;q^5)_{\infty}},
\tag{\textbf{S20}}\\
(-q;q^2)_{\infty}\sum_{n=0}^{\infty}\frac{q^{n^2+n}}{(q^2;q^2)_n(-q;q^2)_{n}}
&=\frac{1}{(q,q^4;q^5)_{\infty}}. \tag{\textbf{S99}}
\end{align}}

The equality of the three left sides of  these equations once again
easily follow from  the summation formulae in Lemmas \ref{l2} and
\ref{l3}. The identity \textbf{S18} (\emph{The First
Rogers-Ramanujan Identity}) also has a well-known interpretation in
terms of partitions, namely, that if $A(n)$ denotes the number of
partitions of $n$ into distinct parts with a gap of at least 2
between consecutive parts, and $B(n)$ denotes the number of
partitions of $n$ into parts $\equiv 1, 4 (\mod 5)$, then
$A(n)=B(n)$ for all positive integers $n$.

 As with the previous three identities, Lemma
\ref{l1} implies two new partition identities.

\begin{theorem}
For a positive integer $n$, let $A(n)$ denotes the number of
partitions of $n$ into distinct parts  a gap of at least 2 between
consecutive parts, and let $B(n)$ denote the number of partitions of
$n$ into parts $\equiv 1, 4 (\mod 5)$.

Let $C(n)$ denote the number of partitions of $n$ into distinct
parts, such that if $o_j$ is the j-th odd part (where we order the
parts in ascending order), then the even parts $o_j+2j-3$ and $o_j
+2j-1$ do not appear.

Let $D(n)$ denote the number of partitions of $n$ into distinct
parts, such that if $e_j$ is the j-th even part (where again we
order the parts in ascending order), then the odd parts $e_j+2j-3$
and $e_j +2j-1$ do not appear.

Then
\begin{equation}
A(n)=B(n)=C(n)=D(n).
\end{equation}
\end{theorem}

\begin{proof}
Once again, all that is necessary is to show that
{\allowdisplaybreaks
\begin{align*}
(-q^2;q^2)_{\infty}\sum_{n=0}^{\infty}\frac{q^{n^2}}{(q^4;q^4)_n}
&=\sum_{n=0}^{\infty}C(n)q^n\\
(-q;q^2)_{\infty}\sum_{n=0}^{\infty}\frac{q^{n^2+n}}
{(q^2;q^2)_n(-q;q^2)_{n}}
&=\sum_{n=0}^{\infty}D(n)q^n.
\end{align*}}

As in the proof of the previous theorem, we do this for the second identity only, since the proof for the former follows similarly. By Lemma \ref{l1}, with $q$ replaced with $q^2$, $m=1$ and $b=1/q$,
\begin{align*}
(-q;q^2)_{\infty}&\sum_{n=0}^{\infty}\frac{q^{n^2+n}}{(q^2;q^2)_n(-q;q^2)_{n}}
\\
&=\sum_{n=0}^{\infty}\sum_{0\leq a_1\leq a_2\leq \dots \leq a_n}\frac{q^{(2+2a_1)+(4+2a_2)+\dots+
(2n+2a_n)}(-q;q^2)_{\infty} } {\prod_{j=1}^{n}(1+q^{4j-3+2a_{j}})(1+q^{4j-1+2a_{j}})}.
\end{align*}
This last series is the generating function for partitions into distinct parts,  such that if $2j+2a_j$ is the $j$-th even part, then the odd parts $(2j+2a_j) +2j-3$ and $(2j+2a_j) +2j-1$ do not appear. This is precisely the partitions of an integer $n$ counted by $D(n)$.
\end{proof}

This time, as an example, we consider the six partitions of 10 counted by $B(10)$, $C(10)$ and $D(10)$. Those counted by $B(10)$ are
\begin{multline*}
\{1, 1, 1, 1, 1, 1, 1, 1, 1, 1\}, \{4, 1, 1, 1, 1, 1, 1\}, \{4, 4, 1, 1\}, \\ \{6, 1
, 1, 1, 1\}, \{6, 4\}, \{9, 1\},
\end{multline*}
those counted by $C(10)$ are
\begin{equation*}
\{5, 4, 1\}, \{6, 4\}, \{7, 3\}, \{8, 2\}, \{9, 1\}, \{10\},
\end{equation*}
while those counted by $D(10)$ are
\begin{equation*}
\{6, 3, 1\}, \{6, 4\}, \{7, 3\}, \{8, 2\}, \{9, 1\}, \{10\}.
\end{equation*}

Note that $D(10)$ does not count $\{5, 4, 1\}$ (since $e_1=4$ and $e_1+2(1)-1=5$), while $C(10)$ does not count $\{6, 3, 1\}$ (since $o_2=3$ and $o_2+2(2)-1=6$).

\subsection{The Rogers-Selberg Identities}
Before coming to the Rogers - Selberg identities, we recall that the \emph{union} of the partitions $\pi$ and $\lambda$, denoted $\pi \cup \lambda$, is the partition whose parts
are those of $\pi$ and $\lambda$ together, arranged in non-increasing order. For example,
\[
\{4,3,3,2,2,1\}\cup\{5,4,3,2,2,1,1\}=\{5,4,4,3,3,3,2,2,2,2,1,1,1\}.
\]
A \emph{bipartition} of a positive integer $n$ is an ordered pair of partitions $(\pi,\lambda)$ such that $\pi \cup \lambda$ is a partition of $n$. Note that $\pi$ or  $\lambda$ may be empty.

The following identity was proved by Rogers \cite{R94} and also later  by Selberg \cite{S36} and Slater \cite{S52}:
\begin{equation}\label{mod71}
(-q;q)_{\infty}\sum_{n=0}^{\infty}\frac{q^{2n^2}}{(q^2;q^2)_n(-q;q)_{2n}}
=\frac{1}{(q,q^2,q^5,q^6;q^7)_{\infty}}. \tag{\textbf{S33}}
\end{equation}
We may interpret this identity as follows.

\begin{theorem}
For a positive integer $n$, let $A(n)$ denote the number of partitions of $n$ into parts $\equiv \pm 1, \pm 2 (\mod 7)$.

Let $B(n)$ denote the number of bipartitions $(\pi,\lambda)$ of $n$, where $\pi$ is a partition into distinct even parts with a gap of at least 4 between consecutive parts, and $\lambda$ is a partition into distinct parts such that if $e_j$ is the $j$-th part in $\pi$ (where we order the parts in ascending order), then the parts $e_j/2$, $e_j/2+1$ and $e_j/2+2$ are not present in $\lambda$.

Let $C(n)$ denote the number of bipartitions $(\pi,\mu)$ of $n$, where $\pi$ is as above, and $\mu$ is a partition into distinct parts such that if $e_j$ is the $j$-th part in $\pi$ (where,as above, we order the parts in ascending order), then the parts $e_j/2+j-1$, $e_j/2+j$ and $e_j/2+j+1$ are not present in $\mu$.

Then
\[
A(n)=B(n)=C(n).
\]
\end{theorem}

\begin{proof}
The right side of \eqref{mod71} clearly gives
$\sum_{n=0}^{\infty}A(n)q^n$. By Lemmas \ref{l1} and \ref{l2},
respectively, with $b=1$ and $m=2$ in each case,
{\allowdisplaybreaks
\begin{multline*}
(-q;q)_{\infty}\frac{q^{2n^2}}{(q^2;q^2)_n(-q;q)_{2n}} \\=
\sum_{0\leq a_1\leq a_2\leq \dots \leq
a_n}\frac{q^{(2+2a_1)+(6+2a_2)+\dots +(4n-2+2a_n)}(-q;q)_{\infty}}
{\prod_{j=1}^{n}(1+q^{3j-2+a_{j}})(1+q^{3j-1+a_{j}})(1+q^{3j+a_{j}})}\\=
\sum_{0\leq a_1\leq a_2\leq \dots \leq
a_n}'\frac{q^{(2+2a_1)+(6+2a_2)+\dots +(4n-2+2a_n)}(-q;q)_{\infty}}
{\prod_{j=1}^{n}(1+q^{2j-1+a_{j}})(1+q^{2j+a_{j}})(1+q^{2j+1+a_{j}})}.
\end{multline*}
}

Upon noting that the $j$-th addend in the exponent of $q$ in each of
the two multiple sums above is $e_j=4j-2+2a_j$, it can be seen that
summing the first of these sums over all $n$ gives
$\sum_{n=0}^{\infty}C(n)q^n$, while summing the second over all $n$
gives $\sum_{n=0}^{\infty}B(n)q^n$.
\end{proof}

Remark: It is not until $n=14$ do we reach an integer for which there is a difference in the bipartitions counted by $B(n)$ and those counted by $C(n)$: $(\{8,4\}, \{4\})$ is counted by $C(14)$ but not $B(14)$ (since $e_2/2=4$), and $(\{6,2\}, \{6\})$ is counted by $B(14)$ but not $C(14)$ (since $e_2/2+2+1=3+2+1=6$).

A similar analysis (with $b=q$ and $m=2$ in Lemmas \ref{l1} and \ref{l2}) of the next identity, also due independently to Rogers \cite{R17}, Selberg \cite{S36} and Slater \cite{S52},
\begin{equation}\label{mod72}
(-q^2;q)_{\infty}\sum_{n=0}^{\infty}\frac{q^{2n^2+2n}}{(q^2;q^2)_n(-q^2;q)_{2n}}
=\frac{1}{(q^2,q^3,q^4,q^5;q^7)_{\infty}}, \tag{\textbf{S31}}
\end{equation}
leads to the following partition interpretation.

\begin{theorem}
For a positive integer $n$, let $A(n)$ denote the number of partitions of $n$ into parts $\equiv \pm 2, \pm 3 (\mod 7)$.

Let $B(n)$ denote the number of bipartitions $(\pi,\lambda)$ of $n$, where $\pi$ is a partition into distinct even parts greater than 2 with a gap of at least 4 between consecutive parts, and $\lambda$ is a partition into distinct parts greater than 1 such that if $e_j$ is the $j$-th part in $\pi$ (where we order the parts in ascending order), then the parts $e_j/2$, $e_j/2+1$ and $e_j/2+2$ are not present in $\lambda$.

Let $C(n)$ denote the number of bipartitions $(\pi,\mu)$ of $n$, where $\pi$ is as above, and $\mu$ is a partition into distinct parts greater than 1 such that if $e_j$ is the $j$-th part in $\pi$ (where,as above, we order the parts in ascending order), then the parts $e_j/2+j-1$, $e_j/2+j$ and $e_j/2+j+1$ are not present in $\mu$.

Then
\[
A(n)=B(n)=C(n).
\]
\end{theorem}

Remark: In this case it is not until $n=19$ do we reach an integer for which there is a difference in the bipartitions counted by $B(n)$ and those counted by $C(n)$: $(\{10,4\}, \{5\})$ is counted by $C(19)$ but not $B(19)$ (since $e_2/2=5$), and $(\{8,4\}, \{7\})$ is counted by $B(19)$ but not $C(19)$ (since $e_2/2+2+1=4+2+1=7$).

Lastly, an analysis (with $b=1$ and $m=2$ in Lemmas \ref{l1} and \ref{l2}) of the remaining Rogers-Selberg-Slater identity (\cite{R17}, \cite{S36} and  \cite{S52}),
\begin{equation}\label{mod73}
(-q;q)_{\infty}\sum_{n=0}^{\infty}\frac{q^{2n^2+2n}}
{(q^2;q^2)_n(-q;q)_{2n}} =\frac{1}{(q,q^3,q^4,q^6;q^7)_{\infty}},
\tag{\textbf{S32}}
\end{equation}
leads to the following result.

\begin{theorem}
For a positive integer $n$, let $A(n)$ denote the number of partitions of $n$ into parts $\equiv \pm 1, \pm 3 (\mod 7)$.

Let $B(n)$ denote the number of bipartitions $(\pi,\lambda)$ of $n$, where $\pi$ is a partition into distinct even parts greater than 2 with a gap of at least 4 between consecutive parts, and $\lambda$ is a partition into distinct parts  such that if $e_j$ is the $j$-th part in $\pi$ (where we order the parts in ascending order), then the parts $e_j/2-1$, $e_j/2$ and $e_j/2+1$ are not present in $\lambda$.

Let $C(n)$ denote the number of bipartitions $(\pi,\mu)$ of $n$, where $\pi$ is as above, and $\mu$ is a partition into distinct parts  such that if $e_j$ is the $j$-th part in $\pi$ (where,as above, we order the parts in ascending order), then the parts $e_j/2+j-2$, $e_j/2+j-1$ and $e_j/2+j$ are not present in $\mu$.

Then
\[
A(n)=B(n)=C(n).
\]
\end{theorem}

This time, it is not until $n=18$ do we reach an integer for which there is a difference in the bipartitions counted by $B(n)$ and those counted by $C(n)$: $(\{10,4\}, \{4\})$ is counted by $C(18)$ but not $B(18)$ (since $e_2/2-1=5-1=4$), and $(\{8,4\}, \{6\})$ is counted by $B(18)$ but not $C(18)$ (since $e_2/2+2=4+2=6$).

\section{Concluding Remarks}
In the bijective part of the hybrid proofs given in the paper,
we have used only the simplest of all bijections, namely, conjugation.
It is likely that other bijections will lead to hybrid proofs of other
basic hypergeometric identities.

The fact that Ramanujan's identity \textbf{Entry 1.4.17} generalizes
the identity in Theorem \ref{t4} (see the remark following Theorem
\ref{t4}) suggests that it may be possible to generalize the
summation formulae in Section 3.


{\allowdisplaybreaks

}

\begin{thebibliography}{99}
\bibitem{A66}
G. E.  Andrews, \emph{$q$-identites of Auluck, Carlitz and Rogers},
Duke Math. J. \textbf{33} (1966). 575-581.

\bibitem{A75}
Andrews, George E.
\emph{Identities in combinatorics. II. A $q$-analog of the Lagrange inversion theorem.}
Proc. Amer. Math. Soc. \textbf{53} (1975), no. 1, 240--245.

\bibitem{A76}
G. E. Andrews, \emph{The Theory of Partitions}, Addison-Wesley, 1976;
Reissued Cambridge, 1998.

\bibitem{A86}
Andrews, G.E., 1986.  $q$-series: their development and application in analysis,
number theory, combinatorics, physics, and computer algebra.
CBMS Regional Conference Series in Mathematics, 66.
Amer. Math. Soc., Providence, RI.

\bibitem{AKK00}
Andrews, George E.; Knopfmacher, Arnold; Knopfmacher, John,
 \emph{Engel expansions and the Rogers-Ramanujan identities.} J. Number Theory \textbf{80} (2000), no. 2, 273--290.

\bibitem{AKP00}
Andrews, George E.; Knopfmacher, Arnold; Paule, Peter.
 \emph{An infinite family of Engel expansions of Rogers-Ramanujan type.} Adv. in Appl. Math. \textbf{25} (2000), no. 1, 2--11.

\bibitem{AB02}
G. E. Andrews, A. Berkovich, \emph{The WP-Bailey
tree and its implications.}
 J. London Math. Soc. (2) \textbf{66} (2002), no. 3, 529--549.


\bibitem{AB09} G.E. Andrews and B.C. Berndt,
\textit{Ramanujan's Lost Notebook, Part II}, Springer, 2009.

\bibitem{BKY10}
B. C. Berndt, B. Kim and A. J. Yee, \emph{Ramanujan's lost notebook:
Combinatorial proofs of identities associated with Heine's
transformation or partial theta functions.} J. Combin. Theory Ser. A
\textbf{117} (2010), no. 7, 957--973.

\bibitem{BA86}
S. Bhargava and C. Adiga, \emph{A basic hypergeometric transformation of Ramanujan and a generalization.} Indian J. Pure Appl. Math. \textbf{17} (1986), no. 3, 338--342.

\bibitem{C73}
L. Carlitz, \emph{Some inverse relations} Duke Math J. \textbf{40}
(1973) 893--901.

\bibitem{CHM08}
Chen, William Y. C., Hou, Qing-Hu and Mu, Yan-Ping,
\emph{Non-terminating basic hypergeometric series and the $q$-Zeilberger algorithm.}
Proc. Edinb. Math. Soc. (2) \textbf{51} (2008), no. 3, 609--633.

\bibitem{C07}
Chu, Wenchang,
 \emph{Abel's lemma on summation by parts and basic hypergeometric series.} Adv. in Appl. Math. \textbf{39} (2007), no. 4, 490--514.


\bibitem{GR90}
G. Gasper and M. Rahman, \emph{An indefinite bibasic summation
formula and some quadratic, cubic and quartic summation and
transformation formulae}, Canad. J. Math. 42 (1990), 1-27.

\bibitem{GR04}
Gasper, George; Rahman, Mizan \emph{Basic hypergeometric series}.
With a foreword by Richard Askey. Second edition. Encyclopedia of
Mathematics and its Applications, 96. Cambridge University Press,
Cambridge, 2004. xxvi+428 pp.


\bibitem{GS83}
Gessel, Ira; Stanton, Dennis.
 \emph{Applications of $q$-Lagrange inversion to basic hypergeometric series.} Trans. Amer. Math. Soc. \textbf{277} (1983), no. 1, 173--201.

\bibitem{GP09}
Gu, Nancy S. S.; Prodinger, Helmut \emph{One-Parameter Generalizations of
Rogers-Ramanujan Type Identities.} Advances in Applied Mathematics
Volume \textbf{45}, Issue 2, August 2010, Pages 149-196.

\bibitem{JS05}
Jouhet, Frédéric; Schlosser, Michael, \emph{Another proof of Bailey's $_6\psi_6$ summation.} Aequationes Math. \textbf{70} (2005), no. 1-2, 43--50.

\bibitem{K93}
Koornwinder, Tom H.,
\emph{On Zeilberger's algorithm and its q-analogue.}
Journal of Computational and Applied Mathematics
Volume \textbf{48}, Issues 1-2, 29 October 1993, Pages 91-111.

\bibitem{K96}
Krattenthaler, C.
\emph{A new matrix inverse.}
 Proc. Amer. Math. Soc. \textbf{124} (1996), no. 1, 47--59.

\bibitem{P06}
Pak, Igor.
 \emph{Partition bijections, a survey.}
  Ramanujan J. \textbf{12} (2006), no. 1, 5--75.

\bibitem{R70}
V. Ramamani, \emph{Some identities conjectured by Srinivasa Ramanujan found in his lithographed notes connected with partition theory and elliptic modular functions -- their proofs -- interconnection with various other topics in the theory of numbers and some generalizations thereon}, PhD thesis, University of Mysore, Mysore, 1970

\bibitem{RV72}
V. Ramamani and K. Venkatachaliengar, \emph{On a partition theorem of Sylvester.} Michigan Math. J. \textbf{19} (1972), 137--140.


\bibitem{R94}
L. J. Rogers, \emph{Second memoir on the expansion of certain infinite
products}, {Proc. London Math. Soc.} 25 (1894) 318--343.

\bibitem{R17} L. J. Rogers, \emph{On two theorems of combinatory analysis and some allied
identities}, Proc. London Math. Soc. 16 (1917) 315--336.

\bibitem{S17}
 Schur, Issai,
\emph{Ein Beitrag zur additiven Zahlentheorie und zur Theorie der
Kettenbr\"{u}chen}, in
\emph{ Gesammelte Abhandlungen. Band II},
Springer-Verlag, Berlin-New York, 1973, 117-136.
(Originally in \emph{Sitzungsberichte der Preussischen Akadamie der
Wissenschaften}, 1917, Physikalisch-Mathematische Klasse, 302-321)

\bibitem{S36}
A. Selberg. \emph{\"{U}ber einige arithmetische Identit\"{a}ten.} Avrandlinger Norske Akad., \textbf{8},
1936.


\bibitem{S03}
Sills, A.V., 2003.
Finite Rogers-Ramanujan type identities,
Electronic J. Combin. 10 \#R13, 122 pp.

\bibitem{S52}
L. J. Slater, \emph{Further identities of the {R}ogers-{R}amanujan type},
{Proc. London Math. Soc.} 54 (1952) 147--167.


\bibitem{S87}
H. M. Srivastava, \emph{A note on a generalization of a $q$-series transformation of Ramanujan.} Proc. Japan Acad. Ser. A Math. Sci. \textbf{63} (1987), no. 5, 143--145.

\bibitem{W03}
S. O. Warnaar, \emph{Extensions of the well-poised and elliptic
well-poised Bailey lemma.} Indag. Math. (N.S.) \textbf{14} (2003),
no. 3-4, 571--588.

\bibitem{W37}
G. N. WATSON, \emph{A note on Lerch's functions}, Quarterly Journal
of Mthematics, Oxford. Series, vol. \textbf{8} (1937), pp. 43-47.

\end{thebibliography}
\end{document}